\documentclass[10pt]{article}
 \usepackage{amsmath, amsfonts, amsthm, amssymb, amscd, enumerate, url, slashed, thmtools, stmaryrd, upgreek, mathdots, mathtools}
\usepackage{bm}
 \usepackage{float}
 \usepackage{tikz}
 \usetikzlibrary{matrix}
 
\begingroup 
 \makeatletter
 \@for\theoremstyle:=definition,remark,plain\do{%
 \expandafter\g@addto@macro\csname th@\theoremstyle\endcsname{%
 \addtolength\thm@preskip\parskip
 }%
 }
 \endgroup
 
\usepackage[titletoc,toc,title]{appendix}
 
 \usepackage[colorlinks = true,
 linkcolor = blue,
 urlcolor = cyan,
 citecolor = red,
 anchorcolor = blue,
 pagebackref]{hyperref}
 
\usepackage[alphabetic,backrefs]{amsrefs}
 
 \usepackage[parfill]{parskip}
 \usepackage{anysize}
\marginsize{1in}{1in}{1in}{1in}

\declaretheorem[name=Theorem,numberwithin=section]{thm}
\declaretheorem[name=Proposition,numberlike=thm]{prop}
\declaretheorem[name=Lemma,numberlike=thm]{lemma}
\declaretheorem[name=Corollary,numberlike=thm]{cor}

\declaretheorem[name=Definition,style=definition,qed=$\blacktriangle$,numberlike=thm]{defn}

\declaretheorem[name=Remark,style=definition,qed=$\blacktriangle$,numberlike=thm]{rmk}

\newcommand{\VF}{\mathfrak{X}}

\newcommand{\w}{\wedge}

\newcommand{\U}[1]{\mathrm{U}(#1)}
\newcommand{\SU}[1]{\mathrm{SU}(#1)}
\newcommand{\Sp}[1]{\mathrm{Sp}(#1)}

\newcommand{\R}{\mathbb R}
\newcommand{\C}{\mathbb C}

\newcommand{\del}{\partial}
\newcommand{\delbar}{\overline{\partial}}

\newcommand{\ol}{\overline}

\newcommand{\mubar}{\overline{\mu}}
\newcommand{\br}[2]{\llbracket #1, #2 \rrbracket}
 \newcommand{\bu}{\bullet}

\begin{document}

\title{Hodge theoretic results for nearly K\"ahler manifolds \\ in all dimensions}

\author{\begin{tabular}{l} Michael Albanese${}^1$, \, \tt{michael.albanese@adelaide.edu.au} \\ Spiro Karigiannis${}^2$, \, \tt{karigiannis@uwaterloo.ca} \\ Luc\'ia Mart\'in-Merch\'an${}^3$, \, \tt{lucia.martin.merchan@hu-berlin.de} \\ Aleksandar Milivojevi\'c${}^2$, \, \tt{amilivojevic@uwaterloo.ca} \\ \, \\ ${}^1$\emph{School of Computer and Mathematical Sciences}, \emph{University of Adelaide} \\ ${}^2$\emph{Department of Pure Mathematics}, \emph{University of Waterloo} \\ ${}^3$\emph{Institut f\"ur Mathematik}, \emph{Humboldt Universit\"at zu Berlin} \end{tabular}}

\date{February 5, 2026}

\maketitle

\begin{abstract}
We generalize to nearly K\"ahler manifolds of arbitrary dimensions most of the Hodge-theoretic results for nearly K\"ahler $6$-manifolds that were established by Verbitsky. In particular, for a compact nearly K\"ahler manifold of any dimension, the (appropriately defined) Hodge numbers are related to the Betti numbers in the same way as on a compact K\"ahler manifold. In the $6$-dimensional case, Verbitsky was able to say slightly more using the induced $\SU{3}$ structure. We discuss potential extensions of this to twistor spaces over positive scalar curvature quaternionic-K\"ahler manifolds, which are a particular class of $(4n+2)$-dimensional nearly K\"ahler manifolds equipped with a special $\SU{n} \! \cdot \! \U{1}$ structure.
\end{abstract}

\tableofcontents

%%%%%%%%%%%%
\section{Introduction} \label{sec:intro}
%%%%%%%%%%%%

Among the various classes of almost Hermitian manifolds, those which are nearly K\"ahler play a special role. They are defined by the condition that $\nabla \omega$ is totally skew-symmetric, which, in the Gray--Hervella classification~\cite{GH}, corresponds to only one particular irreducible component in the decomposition of the intrinsic torsion being nonzero.

In dimension six, strict nearly K\"ahler manifolds (those with $\nabla \omega \neq 0$) are particularly important but rare. To date, only four manifolds are known to admit such structures: $S^6$, $S^3 \times S^3$, $\mathbb{CP}^3$ and $F_{1,2}=\mathrm{SU}(3)/T^2$. Non-homogeneous nearly K\"ahler metrics have been constructed only on the first two by Foscolo--Haskins~\cite{FoscoloHaskins}. In the latter two, the nearly K\"ahler structure arises by viewing them as the twistor space of the anti-self-dual manifolds $S^4$ and $\mathbb{CP}^2$, respectively.

Strict nearly K\"ahler $6$-manifolds admit a compatible $\mathrm{SU}(3)$ structure encoded by a nowhere vanishing $(3,0)$-form $\Theta$ satisfying 
\begin{equation} \label{eq:NK6}
d\omega= 3 \lambda \mathrm{Re}(\Theta), \quad d\mathrm{Im}(\Theta)=-2\lambda \omega^2, \qquad \lambda \in \R\backslash \{0\},
\end{equation}
which leads to additional geometric properties. For instance, they admit a unit-length real Killing spinor and hence, an Einstein metric. (See~\cite{Russo}, for example, for some of these details.) Another consequence is that cones over them carry a parallel $\mathrm{G}_2$ structure, and thus they play an essential role in the study~\cite{K-desings, KL-conifolds} of torsion-free $\mathrm{G}_2$ conifolds.

Higher dimensional nearly K\"ahler manifolds are also linked to special holonomy, as these appear as one of the canonical structures on the twistor space of a positive scalar curvature quaternionic-K\"ahler manifold. See~\cite[Theorem 14.3.9]{BG} for a precise statement, or~\cite{AGI, Nagy} for details.
 
Nearly K\"ahler manifolds carry a metric-compatible Hermitian connection $\nabla^c$ with totally skew-symmetric torsion, given by $T(X,Y)=(\nabla_XJ)JY$. Belgun--Moroianu~\cite[Lemma 2.4]{BelgunMoroianu} showed that $\nabla^cT=0$. This property imposes quite strong restrictions on either the holonomy of $\nabla^c$ or the symmetries of the manifold~\cite{Cleyton}. Nagy~\cite{Nagy} exploited this fact to establish a classification of strict, complete, and simply connected nearly K\"ahler manifolds. Any such manifold is a Riemannian product of three types of building blocks: twistor spaces of positive scalar curvature quaternionic-K\"ahler manifolds, general six-dimensional strict nearly K\"ahler manifolds, and certain symmetric examples.

The present paper builds upon earlier work of Verbitsky on strict nearly K\"ahler $6$-manifolds~\cite{Verbitsky}, where exploiting the $\mathrm{SU}(3)$ structure, he demonstrated a generalization of the K\"ahler identities. Using these, he then proved a Hodge decomposition for the cohomology analogous to the K\"ahler case, together with vanishing statements for some corresponding Hodge numbers on strict nearly K\"ahler $6$-manifolds. Our main contribution is to extend these identities and the Hodge decomposition to nearly K\"ahler manifolds in arbitrary dimensions. While there exist more general analogues of the K\"ahler identities for Hermitian manifolds~\cite{FernandezHosmer}, we obtain stronger conclusions in this particular case by working directly with operators on the exterior algebra.

More broadly, this work identifies some intrinsic properties of nearly K\"ahler geometry, independent of Nagy's classification result. The motivation for this arose while attempting to find an intrinsic proof of formality for compact nearly K\"ahler manifolds, which was established in~\cite{AmannTaimanov} using the Nagy classification. The hope is that, as in the compact K\"ahler case, these general nearly K\"ahler identities might eventually be used to establish an analogue of the $\partial \ol{\partial}$-lemma, leading to a direct proof of formality as in the K\"ahler case~\cite{DGMS}.

The paper is structured as follows: after introducing the necessary background in section~\ref{sec:prelim}, we generalize Verbitsky's nearly K\"ahler identities to all dimensions in section~\ref{sec:nk-id}. We then derive some consequences from these in section~\ref{sec:consequences}, and finally establish a Hodge decomposition valid beyond dimension six in section~\ref{sec:hodge}. This allows us to recover Verbitsky's Hodge number vanishing result for strict nearly K\"ahler $6$-manifolds. A summary and discussion are presented in Section~\ref{sec:conclusion}.

\textbf{Acknowledgements.} The authors thank the anonymous referees for their feedback which improved our paper. Author SK is funded by a Discovery Grant from the Natural Sciences and Engineering Research Council of Canada (NSERC). Author LM is funded by the Deutsche Forschungsgemeinschaft (DFG, German Research Foundation) under Germany's Excellence Strategy -- The Berlin Mathematics Research Center MATH+ (EXC-2046/1, EXC-2046/2, project ID: 390685689).

%%%%%%%%%%%%
\section{Preliminaries} \label{sec:prelim}
%%%%%%%%%%%%

In this section we discuss linear operators on the space of forms and give a review of basic facts about nearly K\"ahler manifolds.

%%%%%%
\subsection{Linear operators on the space of forms} \label{subsec:diff-ops}
%%%%%%

Let $(M, g)$ be a Riemannian manifold. Given a linear operator $P$ on $\Omega^{\bu}(M)$, we say that it is a \emph{graded} operator of degree $|P|$ if $P$ restricts to $P \colon \Omega^k(M) \to \Omega^{k + |P|}(M)$ for every $k$. Its metric adjoint $P^*$ is a graded operator of degree $-|P|$. For example, the operator $L_{\beta}$ given by left multiplication by a form $\beta \in \Omega^{\ell}(M)$, that is $L_{\beta} \gamma = \beta \w \gamma$, has degree $\ell$, and its metric adjoint $\Lambda_{\beta} = (L_{\beta})^*$ has degree $-\ell$.
 
Given two graded operators $P$, $Q$, their graded commutator is
$$ \br{P}{Q} = P \circ Q - (-1)^{|P||Q|} Q \circ P. $$
Note that the graded commutator satisfies $\br{P}{Q}^* = \br{Q^*}{P^*}$, and the graded Jacobi identity
$$ \br{P}{\br{Q}{R}} = \br{\br{P}{Q}}{R} + (-1)^{|P||Q|} \br{Q}{\br{P}{R}}. $$ 
If either of the operators $P,Q$ has even degree, then the graded commutator coincides with the classical commutator, and we will occasionally use the classical notation $[P,Q]$ in this case.

A \emph{derivation} is a graded operator $D$ that satisfies the graded Leibniz rule:
$$ D(\alpha \wedge \beta) = (D \alpha) \wedge \beta + (-1)^{|D| |\alpha|} \alpha \wedge (D \beta). $$
This can be expressed concisely as $\br{D}{L_\alpha} = L_{D\alpha}$ for all $\alpha$.

In addition, we have the notion of the \emph{algebraic order} of a linear operator on $\Omega^{\bu}(M)$, which is defined recursively as follows. The set of operators with algebraic order $0$ is
$$ \mathcal{A}^0 = \{ L_{\beta} : \beta \in \Omega^{\bu}(M) \}. $$
For $r > 0$, the set of linear operators on $\Omega^{\bu}(M)$ with algebraic order $\leq r$ is
$$ \mathcal{A}^r = \{ P : \br{P}{L_\beta} \in \mathcal{A}^{r-1} \, \, \text{for all $\beta \in \Omega^{\bu}(M)$} \}. $$
It follows that
\begin{equation} \label{eqn:order-bracket}
\br{\mathcal{A}^r}{\mathcal{A}^s} \subseteq \mathcal{A}^{r+s-1}.
\end{equation}
In particular, derivations of degree one have algebraic order $\leq 1$. We shall make use of the following results, whose proofs can be found in~\cite{Verbitsky}.

\begin{prop} \label{lem:order-Lambda}\cite[Proposition 2.7]{Verbitsky}
If $\beta \in \Omega^k(M)$, then $\Lambda_{\beta} = (L_{\beta})^*$ has algebraic order $\leq k$.
\end{prop}

\begin{prop} \label{prop:order-adjoint-P}~\cite[Proposition 2.8]{Verbitsky}
If $P$ has degree $1$ and algebraic order $\leq 1$, then its metric adjoint $P^*$ has algebraic order $\leq 2$.
\end{prop}

\begin{prop}\label{prop:order-1-deg-1}\cite[Remark 2.4]{Verbitsky}
 Operators with algebraic order $\leq 1$ and degree $1$ are completely determined by their values on functions and $1$-forms.
\end{prop}

%%%%%%
\subsection{Nearly K\"ahler manifolds} \label{subsec:nk-prelim}
%%%%%%

Let $(M, g, J)$ be an almost Hermitian $2n$-dimensional manifold, and let $\omega \in \Omega^2(M)$ be the fundamental $2$-form $\omega(X, Y) = g(JX, Y)$. Given a complex-valued vector field $X \in \VF_{\C}(M)$, we denote by $X_{1,0}$ and $X_{0,1}$ its $(1,0)$ and $(0,1)$ components, namely $X_{1,0} = \frac{1}{2}(X - i JX)$ and $X_{0,1} = \frac{1}{2}(X + i J X)$. In the same way, given a complex-valued $k$-form $\alpha \in \Omega^k_{\C}(M)$, we denote its projection to $\Omega^{p,q}(M)$ by $\alpha^{p,q}$. The extension of the exterior derivative $d$ to the complex-valued forms $\Omega^{\bu}_{\C}(M)$ decomposes as
$$ d = \mu + \del + \delbar + \mubar. $$
These operators have bi-degrees $(2,-1)$, $(1,0)$, $(0,1)$, and $(-1,2)$, respectively. The equation $d^2 = 0$ yields
\begin{align}
\mu^2 & = 0, & \mubar^2 & = 0, \label{eqn:br-1} \\
\br{\del}{\mu} & = 0, & \br{\delbar}{\mubar} & = 0, \label{eqn:br-2} \\
\br{\delbar}{\mu} & = - \del^2, & \br{\del}{\mubar} & = - \delbar^2, \label{eqn:br-3}
\end{align}
and
\begin{equation}
\br{\del}{\delbar} + \br{\mu}{\mubar} = 0. \label{eqn:br-4}
\end{equation}
Note that for graded operators $P$ and $Q$, we have $\ol{\br{P}{Q}} = \br{\ol{P}}{\ol{Q}}$ where $\ol{P}\alpha := \ol{P\ol{\alpha}}$ for $\alpha \in \Omega^{\bu}_{\C}(M)$.

A bundle endomorphism $A$ of the complexified tangent bundle acts on complex differential forms as follows: for $\alpha \in \Omega^k_{\C}(M)$, define $A\alpha \in \Omega^k_{\C}(M)$ by
$$ (A\alpha)(X_1, \dots, X_k) := \alpha(AX_1, \dots, AX_k). $$
Note, it follows that $A(\alpha\wedge\beta) = (A\alpha)\wedge(A\beta)$. The main case of interest is $A = J$ which satisfies $J\alpha^{p,q} = i^{p-q}\alpha^{p,q}$. With this in mind, we consider the $J$-twisted exterior derivative
\begin{equation}
d^c = J^{-1} d J = i (\mu - \del + \delbar - \mubar). \label{eqn:d^c}
\end{equation}

Let $N$ be the Nijenhuis tensor of the almost complex structure $J$, which is the $TM$-valued $2$-form
$$ N(X,Y) = [X,Y] + J [JX,Y] + J[X,JY] - [JX,JY]. $$
In terms of the Levi-Civita connection $\nabla$ associated with the metric $g$, the Nijenhuis tensor can be expressed as:
\begin{equation} \label{eqn:ni}
N(X,Y)= J(\nabla_X J)Y - (\nabla_{JX} J)Y - (\, J(\nabla_Y J)X - (\nabla_{JY} J)X \,). 
\end{equation}
This uses $[X,Y]=\nabla_X Y-\nabla_Y X$ and $(\nabla_X J)Y = \nabla_X (JY)- J \nabla_X Y$ (see for instance~\cite[proof of Lemma 11.4]{Moroianu}).

Given vector fields $X, Y\in \VF(M)$, it is well-known, and easy to verify, that
$$ [X_{1,0},Y_{1,0}]_{0,1} = \frac{1}{8} (N(X,Y) + i J N(X,Y)), \qquad
[X_{0,1}, Y_{0,1}]_{1,0} = \frac{1}{8} (N(X,Y) - i J N(X,Y)). $$
The next computation shows that the real part of $\mu$ is proportional to the Nijenhuis tensor, viewed as a graded derivation. The latter is the extension of $N\colon \Omega^1(M) \to \Omega^2(M)$, $(N\alpha)(X,Y)=\alpha(N(X,Y))$ to $\Omega^{\bu}(M)$ as a degree 1 derivation. Given $\alpha\in \Omega^{1}(M)$, Cartan's formula implies:
\begin{align*}
\mubar(\alpha^{1,0})(X_{0,1},Y_{0,1}) &= 
d\alpha^{1,0}(X_{0,1},Y_{0,1})\\
&= X_{0,1} (\alpha^{1,0}(Y_{0,1}))- Y_{0,1}(\alpha^{1,0}(X_{0,1})) - \alpha^{1,0}([X_{0,1},Y_{0,1}])\\
&= -\alpha^{1,0}([X_{0,1},Y_{0,1}]_{1,0})\\
&= -\alpha([X_{0,1},Y_{0,1}]_{1,0})\\
 & = -\frac{1}{8}\alpha(N(X,Y) - iJ N(X,Y)).
\end{align*}
In the same way, $\mu({\alpha}^{0,1})(X_{1,0},Y_{1,0})= -\frac{1}{8}\alpha(N(X,Y) + iJ N(X,Y))$. From this we deduce
\begin{align}
(\mu + \mubar)(\alpha)(X,Y) & = -\frac{1}{4} \alpha(N(X,Y))=-\frac{1}{4}(N\alpha)(X,Y), \label{eqn:mu+mubar} \\
(\mu - \mubar)(\alpha)(X,Y) & = -\frac{i}{4} \alpha(JN(X,Y))= -\frac{i}{4}(N (J\alpha))(X,Y) . \label{eqn:mu-mubar}
\end{align}
Finally, since both $\mu + \mubar$ and $N$ are graded derivations that vanish on functions, we obtain
\begin{equation} \label{eqn:N-mu-mubar}
N=-4(\mu + \mubar)
\end{equation}
by Proposition~\ref{prop:order-1-deg-1}.

We also consider the Leftschetz operator $L := L_{\omega}$ and its metric adjoint $\Lambda := L_{\omega}^*$. It is well-known (see~\cite[Proposition 1.2.26]{Huybrechts}) that $[L,\Lambda]=H$, where $H$ is the counting operator,
$$ H=\sum_{k=0}^n(k-n)\pi_{\Omega^k(M)}. $$
In addition, $[H,L]=2L$ and $[H,\Lambda]=-2\Lambda$. 

\begin{defn}
 An almost Hermitian manifold $(M,g,J)$ is \emph{nearly K\"ahler} if
\begin{equation}\label{eq:adef-nk}
(\nabla_X J)(Y)=-(\nabla_YJ)(X) \mbox{ for all } X,Y\in \VF(M)
\end{equation} 
Equation~\eqref{eq:adef-nk} is equivalent to 
\begin{equation}\label{eq:def-nk}
(\nabla_X \omega) (Y,Z)= -(\nabla_Y \omega) (X,Z) \mbox{ for all } X,Y,Z \in \VF(M) .
\end{equation} 
which implies that $\nabla \omega \in \Omega^3(M)$. Therefore, $d\omega= 3 \nabla \omega$.
\end{defn}

\begin{rmk}
The nearly K\"ahler condition is only interesting in real dimension $\geq 6$, because nearly K\"ahler manifolds of real dimension $2$ or $4$ are necessarily K\"ahler. Moreover, the $6$-dimensional case is distinguished, as discussed in the introduction, because it induces an $\SU{3}$ structure. (See also our summary in Section~\ref{sec:conclusion} for an important discussion concerning this point.)
\end{rmk}

The covariant derivative $\nabla \omega$ is a section of a vector bundle $W$ with fibre $\R^{2n}\otimes \mathfrak{u}(n)$. According to the Gray-Hervella classification~\cite{GH}, there is a direct sum decomposition $W=W_1\oplus W_2 \oplus W_3 \oplus W_4$, whose summands in each fiber are irreducible $\mathfrak{u}(n)$-subrepresentations of $\R^{2n}\otimes \mathfrak{u}(n)$, and sections of subbundles of $W$ correspond to almost Hermitian structures whose torsion satisfies certain equations. For example, equation~\eqref{eq:def-nk} describes the subbundle $W_1$. In addition, the Nijenhuis tensor
detects the component in $W_1\oplus W_2$, while $d\omega$ determines the projection onto $W_1\oplus W_3 \oplus W_4$. Due to this, on a nearly K\"ahler manifold, the Nijenhuis tensor, $\nabla \omega$, and $d\omega$ are essentially the same object. The following result reflects this phenomenon:

\begin{lemma} \label{lem:nk-equalities}
Let $(M,g,J)$ be a nearly K\"ahler manifold, and let $\alpha \in \Omega^1(M)$. Then
\begin{align} \label{eqn:N-nabla-J} 
N(X,Y) & = 4J(\nabla_X J)Y,\\ \label{eqn:vanish-deldelbar-om} 
\del \omega & = \delbar \omega=0,\\ \label{eqn:nabla-J-mu-mubar}
\nabla_X (J\alpha) & = - \iota(X)(\mu - \mubar)(i \alpha) + J(\nabla_X \alpha),\\\label{eqn:nabla-om-mu-mubar}
i((\mu-\mubar)\alpha)(X, Y) & = -(\nabla_{\alpha^\#}\omega)(X,Y).
\end{align} Here $\iota(X)(\alpha)$ denotes contraction of the form $\alpha$ by the vector field $X$ and $\alpha^{\#}$ denotes the vector field metric dual to $\alpha$.
\end{lemma}
\begin{proof}
Equation~\eqref{eqn:N-nabla-J} follows from~\eqref{eqn:ni},~\eqref{eq:adef-nk}, and the identity $(\nabla_X J) J = - J  (\nabla_X J) $. We have $(\nabla_{JX} J)Y = - (\nabla_Y J) J X = J (\nabla_Y J) X$, and thus
\begin{align*}
N(X,Y) & = J(\nabla_X J) Y - (\nabla_{JX} J) Y - (\, J(\nabla_Y J) X - (\nabla_{JY} J) X \,) \\
& = 2J(\,(\nabla_X J) Y - (\nabla_Y J) X ) = 4 J (\nabla_XJ)Y.
\end{align*}
In particular, $(\nabla_X \omega) (Y,Z) = g((\nabla_X J)Y, Z)= -\frac{1}{4}\omega(N(X,Y), Z)$. From this, we obtain
\begin{align*}
d\omega(X,Y,Z) &=-\frac{1}{4}(\omega(N(X,Y),Z) + \omega(N(Y,Z),X) + \omega(N(Z,X),Y))\\
&= -\frac{1}{4}(N\omega)(X,Y,Z)\\ 
&= (\mu + \mubar)(\omega)(X,Y,Z),
\end{align*}
which implies equation~\eqref{eqn:vanish-deldelbar-om}. To prove~\eqref{eqn:nabla-J-mu-mubar} it suffices to check $(\nabla_X J)(\alpha)= -\iota(X)(\mu-\mubar)(i\alpha)$ because 
$\nabla_X (J\alpha)= (\nabla_X J)(\alpha) + J(\nabla_X \alpha)$.
This follows from~\eqref{eqn:mu-mubar} and~\eqref{eqn:N-nabla-J}:
$$ ((\nabla_X J)\alpha)(Y)=\alpha ((\nabla_X J)(Y))= -\frac{1}{4}(J\alpha) (N(X,Y)) = -\frac{1}{4}(N(J\alpha))(X, Y) = -i((\mu-\mubar)\alpha)(X,Y). $$
From the above, and the fact that $\nabla \omega \in \Omega^3(M)$, we have
$$ i((\mu-\mubar)\alpha)(X,Y)= -\alpha((\nabla_X J)(Y)) = -g((\nabla_X J)(Y),\alpha^\#)=-(\nabla_X \omega)(Y,\alpha^\#)=-(\nabla_{\alpha^\#} \omega) (X,Y), $$
yielding equation~\eqref{eqn:nabla-om-mu-mubar}.
\end{proof}

Equation~\eqref{eqn:vanish-deldelbar-om} in Lemma~\ref{lem:nk-equalities} gives $\br{L}{\del}=\br{L}{\delbar}=0$ and 
 $\br{L_{\mu\omega}+ L_{\mubar\omega}}{d}=\br{L_{d\omega}}{d}=0$. Equating bi-degrees we get
\begin{align}
&\br{L_{\mu\omega}}{\mu}=\br{L_{\mu\omega}}{\del}=\br{L_{\mu\omega}}{\delbar}=\br{L_{\mubar\omega}}{\mubar}=\br{L_{\mubar\omega}}{\delbar}=\br{L_{\mubar\omega}}{\del}=0, \label{eqn:br-6} \\
&
\br{L_{\mu\omega}}{\mubar} + \br{L_{\mubar\omega}}{\mu}=0 \label{eqn:br-7}.
\end{align}

\begin{rmk}
It follows easily from equation~\eqref{eqn:nabla-om-mu-mubar} that a nearly K\"ahler manifold is actually K\"ahler if and only if $\mu = 0$. When $\mu \neq 0$, we call it a \emph{strict} nearly K\"ahler manifold. All the results we derive in this paper reduce to the classical results of K\"ahler geometry when $\mu = 0$, which forces also that $d \omega = 0$.
\end{rmk}

%%%%%%%%%%%%
\section{Nearly K\"ahler identities} \label{sec:nk-id}
%%%%%%%%%%%%

The classical K\"ahler identities can be written compactly as $[d^*,L]=-d^c$. One way to prove this is to express $[d^*,L]$ and $-d^c$ in terms of a local oriented orthonormal frame $(e_1,\dots, e_{2n})$. It turns out that they both coincide with
$$ \sum_{j=1}^{2n}{Je^i \wedge \nabla_{e_i}(\cdot)}, $$
where $e^i$ is metric dual to $e_i$. As we show below in Theorem~\ref{thm:nk-id}, on a nearly K\"ahler manifold, the difference between $[d^*,L]$ and $-d^c$ is an operator determined by $\mu$ and $\mubar$. This is obtained by comparing both $[d^*,L]$ and $-d^c$ with $\sum_{j=1}^{2n}{Je^i\wedge \nabla_{e_i}(\cdot)}$. The comparison of the latter two operators is carried out in the following lemma.

\begin{lemma}\label{lem:d^c}
Let $(M,g,J)$ be a nearly K\"ahler manifold and let $(e_1, \dots, e_{2n})$ be a local oriented orthonormal frame.
Then
$$ d^c \beta =- \sum_{j=1}^{2n}{Je^j\wedge \nabla_{e_j}\beta} + 2i (\mu-\mubar)\beta. $$
\end{lemma}
\begin{proof}
Both $d^c$ and the right-hand side operator are degree $1$ derivations, so by Proposition~\ref{prop:order-1-deg-1} it suffices to verify the identity on functions and $1$-forms. In what follows, we use the facts that $d\beta= \sum_{j=1}^{2n}{e^j \wedge \nabla_{e_j}\beta}$, and $\sum_{j=1}^{2n}e^j \wedge \iota(e_j)\beta= k\beta$ for all $\beta \in \Omega^k(M)$.

If $f\in C^{\infty}(M)$, then $(\mu-\mubar)f=0$ so
$$ d^cf=J^{-1}dJf=J^{-1}df= -J\left(\sum_{j=1}^{2n}{e^j\wedge \nabla_{e_j}f }\right) = -\sum_{j=1}^{2n}{Je^j \wedge \nabla_{e_j}f} + 2i(\mu-\mubar)f. $$

For $\alpha \in \Omega^1(M)$, we have 
$$ d^c\alpha= J^{-1}dJ\alpha= J^{-1}\left(\sum_{j=1}^{2n}{e^j \wedge \nabla_{e_j}(J\alpha)}\right). $$
By~\eqref{eqn:nabla-J-mu-mubar}, $\nabla_{e_j}(J\alpha)=-\iota(e_j)(\mu-\mubar)(i\alpha) + J(\nabla_{e_j}\alpha)$.
The desired identity now follows from the equalities
\begin{align*}
 J^{-1}\left( -\sum_{j=1}^{2n}{e^j \wedge \iota(e_j)(\mu-\mubar)(i\alpha)}\right) & = 
 J^{-1} (-2(\mu-\mubar)(i\alpha))=2i(\mu-\mubar) \alpha, \\
 J^{-1}\left( \sum_{j=1}^{2n}{e^j \wedge J(\nabla_{e_j}\alpha)}\right) & = 
 -\sum_{j=1}^{2n}{Je^j \wedge \nabla_{e_j}\alpha},
\end{align*}
where we have used $(\mu-\mubar)(\alpha)\in \Omega^{2,0}(M)\oplus \Omega^{0,2}(M)$, and $J^{-1}=-\mathrm{Id}$ on $\Omega^{2,0}(M)\oplus \Omega^{0,2}(M)$ in the first equation.
\end{proof}

\begin{thm}[Nearly K\"ahler identities]\label{thm:nk-id}
Let $(M,g,J)$ be a nearly K\"ahler manifold. Then
$$ [d^*,L]=-d^c + 3i(\mu - \mubar) = 2i\mu +i\del - i\delbar -2i\mubar. $$
\end{thm}
\begin{proof}
Since $d$ has algebraic order $\leq 1$ and degree 1, its metric adjoint operator $d^*$ has algebraic order $\leq 2$ by Proposition~\ref{prop:order-adjoint-P}, so $[d^*,L]$ has algebraic order $\leq 1$ and degree $1$ by equation~\eqref{eqn:order-bracket}. Therefore, it suffices to prove the first equality for functions and $1$-forms by Proposition~\ref{prop:order-1-deg-1}. We fix a local oriented orthonormal frame $(e_1,\dots, e_{2n})$, so $d^*\beta=-\sum_{j=1}^{2n}\iota(e_j)(\nabla_{e_j}\beta)$ for any $\beta \in \Omega^\bu (M)$. In the sequel, we use $(\iota(e_j)\omega)(Y)=g(Je_j,Y)=-Je^j(Y)$ and $\iota(e_j)(\nabla_{e_j}\omega)=0$, which follows from~\eqref{eq:def-nk}.

Let $f\in C^\infty(M)$. Using $(\mu-\mubar) f =0$, and Lemma~\ref{lem:d^c}, we obtain
$$ [d^*,L](f)=d^*(f\omega)= -\sum_{j=1}^{2n}{\iota(e_j)((\nabla_{e_j}f) \omega + f \nabla_{e_j}\omega)}= \sum_{j=1}^{2n} {(\nabla_{e_j}f) Je^j} = (-d^c+ 3i(\mu - \mubar))f. $$
For $\alpha \in \Omega^1(M)$, using Lemma~\ref{lem:d^c} and equation~\eqref{eqn:nabla-om-mu-mubar}, we obtain
\begin{align*}
 [d^*,L](\alpha) & = d^*(\omega\wedge\alpha)- \omega\wedge d^*\alpha = -\sum_{j=1}^{2n}{\iota(e_j)(\nabla_{e_j}\omega \wedge\alpha + \omega \wedge \nabla_{e_j}\alpha)}- \omega\wedge d^*\alpha\\
&= -\sum_{j=1}^{2n}[(\iota(e_j)\nabla_{e_j}\omega)\wedge\alpha + \nabla_{e_j}\omega\wedge(\iota(e_j)\alpha) + (\iota(e_j)\omega)\wedge\nabla_{e_j}\alpha + \omega\wedge(\iota(e_j)\nabla_{e_j}\alpha)] - \omega\wedge d^*\alpha\\
&= -\sum_{j=1}^{2n}\alpha(e_j)\nabla_{e_j}\omega - Je^j\wedge\nabla_{e_j}\alpha = -\nabla_{\alpha^{\#}}\omega + \sum_{j=1}^{2n} { Je^j \wedge\nabla_{e_j}\alpha }
 = -d^c\alpha +3i(\mu - \mubar)\alpha.
\end{align*}
The second equality in the statement of the theorem follows immediately from equation~\eqref{eqn:d^c}.
\end{proof}
Equating bi-degrees and taking Hermitian adjoints we obtain the following:
\begin{cor}\label{cor:nk-id}
Let $(M,g,J)$ be a nearly K\"ahler manifold. Then
\begin{align} \label{nk-id-1} 
&[\del^*,L]=-i\delbar,\qquad &&[\del,\Lambda]=-i\delbar^*, \\ \label{nk-id-2}
&[\delbar^*,L]=i\del,\qquad &&[\delbar,\Lambda]=i\del^*,\\ \label{nk-id-3}
&[\mu^*,L]=-2i\mubar,\qquad &&[\mu,\Lambda]=-2i\mubar^*,\\ \label{nk-id-4}
&[\mubar^*,L]=2i\mu,\qquad &&[\mubar,\Lambda]=2i\mu^*.
\end{align}
\end{cor}

\begin{rmk}
Verbitsky~\cite{Verbitsky} proves~\eqref{nk-id-1} and~\eqref{nk-id-2} in all dimensions, and deduces~\eqref{nk-id-3} and~\eqref{nk-id-4} for 6-dimensional nearly K\"ahler manifolds using the corresponding $\SU{3}$ structure~\cite[Proposition 5.1]{Verbitsky}. (There are typographical errors in~\cite[equation (5.1)]{Verbitsky}, where $L_{\omega}$ and $\Lambda_{\omega}$ have been transposed.)
\end{rmk}

The following proposition, which we prove in all dimensions, was proved in the $6$-dimensional case in~\cite[Claim 4.2]{Verbitsky} using the $\SU{3}$ structure, which we do not have in higher dimensions. In this proposition, we investigate the operator $[\Lambda,L_{d\omega}]$, which plays an important role in the ``K\"ahler identities'' for almost Hermitian manifolds for which $d\omega \neq 0$. Its earliest appearance which the authors are aware of is in~\cite{Demailly-identities}. See also~\cite{FernandezHosmer}. Some authors call it the ``torsion operator'', although this can be misleading. When the almost complex structure $J$ is integrable, then the remaining torsion of the $\U{n}$ structure is indeed encoded by $d\omega$, and one can show that $[\Lambda,L_{d\omega}]$ vanishes if and only if $d\omega = 0$. But for non-integrable $J$, this ``torsion operator'' does not encode the full $\U{n}$ structure torsion.

\begin{prop} \label{prop:tor}
Let $(M,g,J)$ be a nearly K\"ahler manifold. Then
$$  [\Lambda,L_{d\omega}]=-3(\mu + \mubar). $$
Therefore, $[\Lambda, L_{\mu\omega}]=-3\mu$, $[\Lambda, L_{\mubar\omega}]=-3\mubar$, $[L_{\mu \omega}^*,L]=-3\mu^*$, and $[L_{\mubar \omega}^*,L]=-3\mubar^*$.
\end{prop}
\begin{proof}
Using $L_{d\omega}=\br{d}{L}$ and the graded Jacobi identity we obtain
$$ \br{\Lambda}{L_{d\omega}}=\br{\Lambda}{\br{d}{L}}=\br{\br{\Lambda}{d}}{L} +\br{d}{\br{\Lambda}{L}}. $$
We work on each term separately.

Theorem~\ref{thm:nk-id} implies $\br{\Lambda}{d}=\br{d^*}{L}^*= -i(2 \mu^* +\del^* - \delbar^* -2\mubar^*)$, and this together with Corollary~\ref{cor:nk-id} then give
$$ \br{\br{\Lambda}{d}}{L}=-i\br{2 \mu^* +\del^* - \delbar^* -2\mubar^*}{L}= -4\mu - \del - \delbar -4\mubar. $$

Next, we have $\br{d}{\br{\Lambda}{L}}=-\br{d}{H}$. If $\alpha \in \Omega^k(M)$, then $\br{d}{H}(\alpha)=(k-n)d\alpha - (k+1-n)d\alpha = -d\alpha$. Hence $\br{d}{\br{\Lambda}{L}} = d$ and thus
$$ \br{\Lambda}{L_{d\omega}}= -4\mu - \del - \delbar -4\mubar + d = -3(\mu + \mubar). $$
Taking components and using~\eqref{eqn:vanish-deldelbar-om} we obtain $[\Lambda, L_{\mu\omega}]=-3\mu$ and $[\Lambda, L_{\mubar\omega}]=-3\mubar$. The remaining two identities follow by taking Hermitian adjoints. 
\end{proof}

\begin{rmk}
Note that Fernandez--Hosmer~\cite{FernandezHosmer} generalized the K\"ahler identities to the almost Hermitian setting. However, to obtain our Theorem~\ref{thm:nk-id}, and Corollary~\ref{cor:nk-id} from their results, one needs the identity in Proposition~\ref{prop:tor}. 
\end{rmk}

\section{Consequences of nearly K\"ahler identities} \label{sec:consequences}

We establish some identities that follow from Corollary~\ref{cor:nk-id} and Proposition~\ref{prop:tor}, when combined with equations~\eqref{eqn:br-1}--\eqref{eqn:br-4},~\eqref{eqn:br-6}, and~\eqref{eqn:br-7}. 

Proposition~\ref{prop:com-laplacian} below provides relations between the commutators that appear as components of the Laplacian $\Delta_d = \br{d^*}{d} = d^*d + dd^*$. This result extends~\cite[Proposition 4.3]{Verbitsky} to arbitrary dimensions. To prove it, we first compute the components of $\br{d^*}{L_{d\omega}}$. The resulting identities generalize some intermediate computations from Verbitsky's proof, to all dimensions.

\begin{lemma} \label{lem:aux-commutators}
Let $(M,g,J)$ be a nearly K\"ahler manifold. Then
\begin{align}
 \br{\mubar^*}{L_{\mu\omega}} & = \br{\delbar^*}{L_{\mu\omega}}=0, \qquad 
 \br{\mu^*}{L_{\mubar\omega}}=\br{\del^*}{L_{\mubar\omega}}=0, \label{eqn:com-lw-1} \\ 
 \br{\delbar}{\mu} & = - \frac{i}{3} \br{\del^*}{L_{\mu \omega}}, \qquad \br{\del}{\mubar}= \frac{i}{3} \br{\delbar^*}{L_{\mubar \omega}}. \label{eqn:com-lw-2} 
\end{align}
\end{lemma}
\begin{proof}
The equations on the right-hand side are complex conjugates of those on the left, so it suffices to verify the latter. Both $\br{\mubar^*}{L_{\mu\omega}} = 0$ and $\br{\delbar^*}{L_{\mu\omega}}=0$ are obtained by successively applying 
 equation~\eqref{eqn:br-1} or~\eqref{eqn:br-2}, Proposition~\ref{prop:tor}, the graded Jacobi identity, Corollary~\ref{cor:nk-id}, and equation~\eqref{eqn:br-6}:
\begin{align*}
0&=\br{\mu}{-3\mu}= \br{\mu}{\br{\Lambda}{L_{\mu\omega}}}
=\br{\br{\mu}{\Lambda}}{L_{\mu \omega}} + \br{\Lambda}{\br{\mu}{L_{\mu \omega}}}= -2i\br{\mubar^*}{L_{\mu \omega}},\\
0 & = \br{\del}{-3\mu}= \br{\del}{\br{\Lambda}{L_{\mu\omega}}}
=\br{\br{\del}{\Lambda}}{L_{\mu \omega}} + \br{\Lambda}{\br{\del}{L_{\mu \omega}}}= -i\br{\delbar^*}{L_{\mu \omega}}.
\end{align*}
Similarly, using Proposition~\ref{prop:tor}, the graded Jacobi identity, Corollary~\ref{cor:nk-id}, and equation~\eqref{eqn:br-6} we obtain
\begin{align*}
\br{\delbar}{\mu} & = -\frac{1}{3}\br{\delbar}{\br{\Lambda}{L_{\mu \omega}}}
=-\frac{1}{3} \br{\br{\delbar}{\Lambda}}{L_{\mu \omega}}-\frac{1}{3}\br{\Lambda}{\br{\delbar}{L_{\mu \omega}}} 
= - \frac{i}{3} \br{\del^*}{L_{\mu \omega}}. \qedhere
\end{align*}
\end{proof}

Note that Lemma~\ref{lem:aux-commutators} yields information about all but two components of $\br{d^*}{L_{d\omega}}$, namely $\br{\mu^*}{L_{\mu\omega}}$ and its conjugate $\br{\bar{\mu}^*}{L_{\bar{\mu}\omega}}$. These will appear in equation~\eqref{eqn:com-lw-A}, on the way to proving Proposition~\ref{proportional}. 

We prove Proposition~\ref{prop:com-laplacian} below using the same method as in~\cite[Proposition 4.3]{Verbitsky}, now with the generalized identities from Corollary~\ref{cor:nk-id} and Proposition~\ref{prop:tor}. This gives information on all possible graded commutators of components of $d$ with components of $d^*$, aside from $\br{P^*}{P}$ where $P \in \{\mu, \partial, \delbar, \mubar\}$. For any operator $P$, we define the \emph{$P$-Laplacian} to be the operator $\Delta_P := \br{P^*}{P}$. Note that only when $P$ has odd order do we have $\Delta_P = P^*P + PP^*$
\begin{prop} \label{prop:com-laplacian}
Let $(M,g,J)$ be a nearly K\"ahler manifold. Then
\begin{align}
\br{\delbar^*}{\mu} & = \br{\del^*}{\mubar}=\br{\mu^*}{\delbar}= \br{\mubar^*}{\del}=0,\\
\br{\mu^*}{\mubar} & = \br{\mubar^*}{\mu}=0,\\
\br{\delbar^*}{\del} & = -\br{\del^*}{\mu}=-\br{\mubar^*}{\delbar}, \\
\br{\del^*}{\delbar} & = -\br{\mu^*}{\del}=-\br{\delbar^*}{\mubar}.
\end{align}
\end{prop}
\begin{proof}
Observe that by taking conjugates and adjoints, it suffices to prove $\br{\delbar^*}{\mu}=0$, $\br{\mubar}{\mu^*}=0$ and $\br{\delbar^*}{\del}=-\br{\del^*}{\mu}$. To verify $\br{\delbar^*}{\mu}=0$, we use Proposition~\ref{prop:tor} followed by the graded Jacobi identity, and finally apply $[\delbar^*,\Lambda] = [L,\delbar]^*= -L_{\delbar\omega}^* = 0$ and equation~\eqref{eqn:com-lw-1}:
$$ -3\br{\delbar^*}{\mu}= \br{\delbar^*}{\br{\Lambda}{L_{\mu\omega}}}= \br{\br{\delbar^*}{\Lambda}}{L_{\mu\omega}} + \br{\Lambda}{\br{\delbar^*}{L_{\mu \omega}}}=0. $$
Using equation~\eqref{nk-id-4} twice, the graded Jacobi identity, and equation~\eqref{eqn:br-1} we obtain
$$ 2i\br{\mubar}{\mu^*}=\br{\mubar}{\br{\mubar}{\Lambda}} = \br{\br{\mubar}{\mubar}}{\Lambda} - \br{\mubar}{\br{\mubar}{\Lambda}} =-2i\br{\mubar}{\mu^*}. $$
Thus $\br{\mubar}{\mu^*}=0$.

To obtain the final identity, first note that $\br{\del}{\del} = 2\del^2 = -2\br{\delbar}{\mu} = \frac{2i}{3}\br{\del^*}{L_{\mu\omega}}$ by equations~\eqref{eqn:br-3} and~\eqref{eqn:com-lw-2}. Now apply $\br{\Lambda}{\cdot\ }$ to both sides. We can simplify the resulting expression on the left hand side by using the graded Jacobi identity and equation~\eqref{nk-id-1}:
$$\br{\Lambda}{\br{\del}{\del}} = \br{\br{\Lambda}{\del}}{\del} + \br{\del}{\br{\Lambda}{\del}} = \br{i\delbar^*}{\del} + \br{\del}{i\delbar^*} = 2i\br{\delbar^*}{\del}.$$

As for $\frac{2i}{3}\br{\Lambda}{\br{\del^*}{L_{\mu\omega}}}$, combining the observation $[\Lambda, \del^*] = [\del, L]^* = L_{\del\omega}^* = 0$ with Proposition~\ref{prop:tor} yields
$$\frac{2i}{3}\br{\Lambda}{\br{\del^*}{L_{\mu\omega}}} = \frac{2i}{3}\br{\br{\Lambda}{\del^*}}{L_{\mu\omega}} + \frac{2i}{3}\br{\del^*}{\br{\Lambda}{L_{\mu\omega}}} = \frac{2i}{3}\br{\del^*}{-3\mu} = -2i\br{\del^*}{\mu}$$
which completes the proof.
\end{proof}

We close this section by showing that the differences of Laplacians $\Delta_{\del}-\Delta_{\delbar}$, $\Delta_{\mu}-\Delta_{\mubar}$, and $\Delta_{L_{\mu\omega}}-\Delta_{L_{\mubar\omega}}$ are all proportional.

\begin{prop}\label{proportional}
Let $(M,g,J)$ be a nearly K\"ahler manifold. Then
\begin{align}
 \br{\mubar}{\mu} & = \frac{i}{9}
\br{\Delta_{L_{\mu\omega}}- \Delta_{L_{\mubar \omega}}}{L}, \label{eqn:com-lw-3} \\
\br{\Lambda}{\br{\mu}{L_{\mu\omega}}} & = - \frac{i}{3} \br{\Delta_{L_{\mu\omega}}+ \Delta_{L_{\mubar \omega}}}{L}, \label{eqn:com-lw-4}\\
\Delta_{\del} - \Delta_{\delbar}= -2(\Delta_{\mu}- \Delta_{\mubar})&= -\frac{2}{9} (\Delta_{L_{\mu\omega}} -\Delta_{L_{\mubar\omega}}). \label{eqn:diff-lapl}
\end{align}
\end{prop}
\begin{proof}
Making use of Proposition~\ref{prop:tor}, the graded Jacobi identity, Corollary~\ref{cor:nk-id}, and equation~\eqref{eqn:br-7} (for the first equality only) we obtain
\begin{align*}
\br{\mubar}{\mu} &= -\frac{1}{3}\br{\mubar}{\br{\Lambda}{L_{\mu \omega}}}
=-\frac{1}{3} \br{\br{\mubar}{\Lambda}}{L_{\mu \omega}}-\frac{1}{3}\br{\Lambda}{\br{\mubar}{L_{\mu \omega}}} 
= - \frac{2i}{3} \br{\mu^*}{L_{\mu \omega}} + \frac{1}{3}\br{\Lambda}{\br{\mu}{L_{\mubar \omega}}} , \\
\br{\mubar}{\mu} &= -\frac{1}{3}\br{\mu}{\br{\Lambda}{L_{\mubar \omega}}}
=-\frac{1}{3} \br{\br{\mu}{\Lambda}}{L_{\mubar \omega}}-\frac{1}{3}\br{\Lambda}{\br{\mu}{L_{\mubar \omega}}} 
= \frac{2i}{3} \br{\mubar^*}{L_{\mubar \omega}} -\frac{1}{3}\br{\Lambda}{\br{\mu}{L_{\mubar \omega}}}.
\end{align*}
Adding and subtracting these, we get
\begin{align}
\br{\mubar}{\mu} & =  - \frac{i}{3} \br{\mu^*}{L_{\mu \omega}} + \frac{i}{3} \br{\mubar^*}{L_{\mubar \omega}}, \label{eqn:com-lw-A} \\
\br{\Lambda}{\br{\mu}{L_{\mubar\omega}}} & = i \br{\mu^*}{L_{\mu \omega}} + i\br{\mubar^*}{L_{\mubar \omega}}. \label{eqn:com-lw-B}
\end{align}
Then Proposition~\ref{prop:tor}, the graded Jacobi identity, and the identity $\br{L_{\mu\omega}}{L}=0$ yield:
$$ -3\br{\mu^*}{L_{\mu\omega}} = \br{\br{L_{\mu\omega}^*}{L}}{L_{\mu\omega}} = \br{L_{\mu\omega}^*}{\br{L}{L_{\mu\omega}}} - \br{L}{\br{L_{\mu\omega}^*}{L_{\mu\omega}}} = -\br{L}{\Delta_{L_{\mu\omega}}} = \br{\Delta_{L_{\mu\omega}}}{L}. $$
Taking conjugates gives $-3\br{\mubar^*}{L_{\mubar{\omega}}}=\br{\Delta_{L_{\mubar \omega}}}{L}$. These two equations, together with~\eqref{eqn:com-lw-A} and~\eqref{eqn:com-lw-B}, yield~\eqref{eqn:com-lw-3} and~\eqref{eqn:com-lw-4}.

To obtain~\eqref{eqn:diff-lapl}, we first apply $\br{\Lambda}{\cdot\ }$ to both sides of equation~\eqref{eqn:com-lw-A}. We can simplify the resulting expression on the left hand side by using the graded Jacobi identity and equations~\eqref{nk-id-3} and~\eqref{nk-id-4}:
$$ \br{\Lambda}{\br{\mubar}{\mu}} = \br{\br{\Lambda}{\mubar}}{\mu} + \br{\mubar}{\br{\Lambda}{\mu}} = \br{-2i\mu^*}{\mu} + \br{\mubar}{2i\mubar^*} = -2i(\Delta_{\mu} - \Delta_{\mubar}). $$
As for $-\frac{i}{3}\br{\Lambda}{\br{\mu^*}{L_{\mu\omega}}} + \frac{i}{3}\br{\Lambda}{\br{\mubar^*}{L_{\mubar\omega}}}$, the second term is the conjugate of the first, so we focus on the first one. The graded Jacobi identity, the identity $[\Lambda, \mu^*] = [\mu, L]^* = L_{\mu\omega}^*$ and Proposition~\ref{prop:tor} yields
$$ \br{\Lambda}{\br{\mu^*}{L_{\mu\omega}}} = \br{\br{\Lambda}{\mu^*}}{L_{\mu\omega}} + \br{\mu^*}{\br{\Lambda}{L_{\mu\omega}}} = \br{L_{\mu\omega}^*}{L_{\mu\omega}} + \br{\mu^*}{-3\mu} = \Delta_{L_{\mu\omega}} - 3\Delta_{\mu}, $$
and similarly $\br{\Lambda}{\br{\mubar^*}{L_{\mubar\omega}}} = \Delta_{L_{\mubar\omega}} - 3\Delta_{\mubar}$. Therefore
$$ -2i(\Delta_{\mu}-\Delta_{\mubar}) = - \frac{i}{3} ( \Delta_{L_{\mu\omega}} -3 \Delta_{\mu} ) + \frac{i}{3} (\Delta_{L_{\mubar\omega}} -3 \Delta_{\mubar}) 
=-\frac{i}{3} (\Delta_{L_{\mu\omega}} - \Delta_{L_{\mubar\omega}}) +i(\Delta_{\mu} - \Delta_{\mubar}), $$
which implies $\Delta_{\mu}- \Delta_{\mubar}= \frac{1}{9} (\Delta_{L_{\mu\omega}} -\Delta_{L_{\mubar\omega}})$.

Finally, expanding the commutator of $\Lambda$ and equation~\eqref{eqn:br-4}, that is $\br{\Lambda}{\br{\del}{\delbar}}=-\br{\Lambda}{\br{\mu}{\mubar}}$, and using Corollary~\ref{cor:nk-id}, we obtain $i\Delta_{\delbar} -i \Delta_{\del}= -2i \Delta_{\mubar} + 2i \Delta_{\mu}$. Thus we obtain
$$ \Delta_{\del} - \Delta_{\delbar}= -2(\Delta_{\mu}- \Delta_{\mubar})=-\frac{2}{9} (\Delta_{L_{\mu\omega}} -\Delta_{L_{\mubar\omega}}). \qedhere $$
\end{proof}

\begin{rmk}
Note that when $\mu = 0$, so that we have a K\"ahler structure, then equation~\eqref{eqn:diff-lapl} reduces to the well-known equality $\Delta_{\del} = \Delta_{\delbar}$ of K\"ahler geometry.
\end{rmk}

\begin{rmk} \label{rem:strict-nk-diff-lapl}
We now argue that equations~\eqref{eqn:com-lw-1} and~\eqref{eqn:com-lw-3} generalize those in~\cite[Proposition 3.2, Corollary 3.3]{Verbitsky}. Let $(M,g, \omega)$ be a nearly K\"ahler 6-manifold. As we discussed in the introduction, there is a complex unit-length volume form $\Theta \in \Omega^{3,0}(M)$ and $\lambda \in \R$ such that $d\omega=3 \lambda \mathrm{Re}(\Theta)$. (See~\cite[Proposition 1.1]{Verbitsky}.) In the two results quoted above, Verbitsky shows
\begin{align}\label{eqn:strict-1} 
\br{\mubar}{\mu}|_{\Omega^{p,q}(M)} &= i\frac{\lambda^2}{4} (p-q) L, \\
\Delta_{\del}- \Delta_{\delbar}|_{\Omega^{p,q}(M)} &= \frac{\lambda^2}{4} (3-p-q)(p-q)\mathrm{Id}.\label{eqn:strict-2}
\end{align}
(Note, in~\cite{Verbitsky}, the constants $\frac{\lambda^2}{4}$ instead appear as $\lambda^2$. This is due to a typographical error in~\cite[equation (1-4)]{Verbitsky}, where $\lambda$ should be replaced by $\frac{\lambda}{2}$. The upshot of this is that every occurrence of $\lambda$ in~\cite{Verbitsky} should be replaced by $\frac{\lambda}{2}$ for comparison with our results).

We first compute $\Delta_{L_{\mu \omega}} - \Delta_{L_{\mubar \omega}}$. In this case, $\mu \omega= \frac{3\lambda}{2} \Theta$, and so $\Delta_{L_{\mu\omega}} = \frac{9\lambda^2}{4} \br{L_{\Theta}^*}{ L_{\Theta}}$. Consider a local oriented orthonormal frame $(e_1,\dots,e_6)$ with $e_{2k}=Je_{2k-1}$, and write $\theta_k = \frac{1}{\sqrt{2}}(e_{2k-1}- i e_{2k})$ and $\theta^{k}=\frac{1}{\sqrt{2}}(e^{2k-1}+ i e^{2k})$. Then $\Theta=\theta^1\wedge \theta^2 \wedge \theta^3$ and $L_{\Theta}^* = \iota(\theta_3)\circ \iota(\theta_2) \circ \iota(\theta_1)$, and a direct computation yields
$$ \br{L_{\Theta}^*}{L_{\Theta}}(\alpha) = \alpha - \sum_{k=1}^3 \theta^k \wedge \iota(\theta_k)(\alpha) + \theta^{12} \wedge \iota(\theta_2)\iota(\theta_1)(\alpha) + \theta^{13} \wedge \iota(\theta_3)\iota(\theta_1)(\alpha) + \theta^{23} \wedge \iota(\theta_3)\iota(\theta_2)(\alpha) $$
where $\theta^{ab} = \theta^a\wedge\theta^b$. If $\alpha \in \Omega^{p,q}(M)$, then by writing $\alpha$ as a sum of simple forms, one can show that $\sum_{k=1}^3 \theta^k \wedge \iota(\theta_k)(\alpha) = p \alpha$ and 
$$ \theta^{12} \wedge \iota(\theta_2)\iota(\theta_1)(\alpha) + \theta^{13} \wedge \iota(\theta_3)\iota(\theta_1)(\alpha) + \theta^{23} \wedge \iota(\theta_3)\iota(\theta_2)(\alpha)= \frac{1}{2}p(p-1)\alpha. $$
Hence, we obtain
$$ \Delta_{L_{\mu \omega}}|_{\Omega^{p,q}(M)}= \frac{9\lambda^2}{4} \left(1-p+ \frac{1}{2}p(p-1)\right)\mathrm{Id}$$
and similarly
$$ \Delta_{L_{\mubar\omega}}|_{\Omega^{p,q}(M)}= \frac{9\lambda^2}{4} \left(1-q+ \frac{1}{2}q(q-1)\right) \mathrm{Id}. $$ Therefore, we deduce that
\begin{align*}
(\Delta_{L_{\mu \omega}} - \Delta_{L_{\mubar \omega}})|_{\Omega^{p,q}(M)} & = \frac{9\lambda^2}{4}\left(1-p+\frac{1}{2}p(p-1) - 1 + q - \frac{1}{2}q(q-1)\right)\mathrm{Id} \\
& = -\frac{9\lambda^2}{8}(3-p-q)(p-q)\mathrm{Id}
\end{align*}
which implies equation~\eqref{eqn:strict-2} due to equation~\eqref{eqn:diff-lapl}. Finally, equation~\eqref{eqn:com-lw-3} together with the above yields
\begin{align*}
\br{\mubar}{\mu}|_{\Omega^{p,q}(M)} & = \frac{i}{9}\br{\Delta_{L_{\mu \omega}}- \Delta_{L_{\mubar \omega}}}{L}|_{\Omega^{p,q}(M)} \\
& = - i \frac{\lambda^2}{8} \big( (3 - (p+1) - (q+1))(p-q) - (3 - p - q)(p-q) \big) \mathrm{Id} \\
& = i \frac{\lambda^2}{4} (p-q)L,
\end{align*}
which is equation~\eqref{eqn:strict-1}.
\end{rmk}

\begin{rmk}
Using the identities that we have established, one can derive a formula for the commutator of $L$ with the Hodge Laplacian $\Delta_d = \br{d^*}{d}$. The graded Jacobi identity gives
$$ \br{L}{\Delta_d} = \br{\br{L}{d^*}}{d} + \br{d^*}{\br{L}{d}}. $$
From Theorem~\ref{thm:nk-id} and equations~\eqref{eqn:br-1}--\eqref{eqn:br-4}, one can deduce that the first term on the right hand side above is $i \delbar^2 - i \del^2$, while using Lemma~\ref{lem:aux-commutators}, the second term is $3 i \del^2 - 3 i \delbar^2 - \br{\mu^*}{L_{\mu \omega}} - \br{\mubar^*}{L_{\mubar \omega}}$. Combining these gives
$$ \br{L}{\Delta_d} = 2 i \del^2 - ( \br{\mu^*}{L_{\mu \omega}} + \br{\mubar^*}{L_{\mubar \omega}} ) - 2 i \delbar^2, $$
decomposed into terms of bidegree $(2,0)$, $(1,1)$, and $(0,2)$ respectively. In particular, if we demand that $\br{L}{\Delta_d} = 0$, then it forces $\delbar^2 = 0$, which is equivalent to integrability of $J$, and we are in the K\"ahler case. The upshot of this observation is that for \emph{strict} nearly K\"ahler manifolds, the Lefschetz decomposition of forms does \emph{not} in general descend to $\Delta_d$-harmonic forms and thus to de Rham cohomology.
\end{rmk}

%%%%%%%%%%%%
\section{Hodge decomposition} \label{sec:hodge}
%%%%%%%%%%%%

Using the relations between the components of $\br{d^*}{d}$, we derive a formula for $\Delta_d$ in terms of other Laplacians, and obtain a Hodge decomposition in the compact case. This result generalizes~\cite[Theorems 6.1, 6.2]{Verbitsky} to all dimensions and uses the same method of proof. 

We define
\begin{align*}
\mathcal{H}^k(M, \C) & = \{\alpha \in \Omega^k_{\C}(M) : \Delta_d\alpha = 0\}, \\
\mathcal{H}^{p,q}(M) & = \{\beta \in \Omega^{p,q}(M) : \Delta_d\beta = 0\}.
\end{align*}
That is, $\mathcal{H}^{p,q}(M)$ consists of the $(p,q)$-forms which are $\Delta_d$-harmonic. This should not be confused with $(p,q)$-forms which are $\Delta_{\delbar}$-harmonic. These two notions agree on K\"ahler manifolds, but not in general. Note that $\Delta_d$ need not preserve the $(p,q)$-type in general.

\begin{thm}\label{thm:nk-hodge}
Let $(M,g,J)$ be a nearly K\"ahler manifold. Then $\Delta_d = \Delta_{\del-\delbar} + \Delta_{\mu} + \Delta_{\mubar}$. 

Moreover, if $M$ is compact then 
\begin{enumerate}
\item[(a)] $\mathcal{H}^k(M, \C) = \bigcap_{P\in\mathcal{P}}\ker (P : \Omega^k_{\C}(M) \to \Omega^{k+1}_{\C}(M))$ where $\mathcal{P} = \{\mu,\del,\delbar,\mubar,\mu^*, \del^*,\delbar^*,\mubar^*\}$, 
\item[(b)] $\mathcal{H}^k(M, \C) = \ker\Delta_{\mu}\cap\ker\Delta_{\del}\cap\ker\Delta_{\delbar}\cap\ker\Delta_{\mubar}$,
\item[(c)] $\mathcal{H}^k(M,\C)=\bigoplus_{p+q=k}\mathcal{H}^{p,q}(M)$, and
\item[(d)] $\mathcal{H}^{q,p}(M) =\ol{\mathcal{H}^{p,q}(M)}$.
\end{enumerate}
In particular, the odd-degree Betti numbers are all even, just as they are for compact K\"ahler manifolds.
\end{thm}
\begin{proof}
Expanding $\Delta_{d}$ and using Proposition~\ref{prop:com-laplacian} we obtain
\begin{align*}
 \Delta_{d} & = \br{d^*}{d}\\ 
 &= \br{\mu^*}{\mu}+ \br{\mu^*}{\del} + \br{\mu^*}{\delbar}+ \br{\mu^*}{\mubar} 
 + \br{\del^*}{\mu}+ \br{\del^*}{\del} + \br{\del^*}{\delbar}+ \br{\del^*}{\mubar} \\
 & \qquad  {} + \br{\delbar^*}{\mu}+ \br{\delbar^*}{\del} + \br{\delbar^*}{\delbar}+ \br{\delbar^*}{\mubar}
 + \br{\mubar^*}{\mu}+ \br{\mubar^*}{\del} + \br{\mubar^*}{\delbar}+ \br{\mubar^*}{\mubar} \\
 & = 
 \br{\mu^*}{\mu} - \br{\del^*}{\delbar} + 0+ 0 
 - \br{\delbar^*}{\del}+ \br{\del^*}{\del} + \br{\del^*}{\delbar}+ 0\\
 & \qquad {} + 0+ \br{\delbar^*}{\del} + \br{\delbar^*}{\delbar}- \br{\del^*}{\delbar}+ 0+ 0 - \br{\delbar^*}{\del}+ \br{\mubar^*}{\mubar} \\
 & = \br{\del^*}{\del} + \br{\delbar^*}{\delbar} - \br{\del^*}{\delbar} - \br{\delbar^*}{\del} + \br{\mu^*}{\mu} +\br{\mubar^*}{\mubar}\\ 
 &= \br{\del^*-\bar\del^*}{\del-\delbar} + \Delta_{\mu} + \Delta_{\mubar}\\
 &= \Delta_{\del-\delbar} + \Delta_{\mu} + \Delta_{\mubar} .
\end{align*}

Now assume that $M$ is compact. We first prove (a). Note that if $\alpha \in \bigcap_{P\in \mathcal{P}}\ker P$, then $d\alpha = 0$ and $d^*\alpha = 0$, so $\Delta_d\alpha = 0$. For the converse direction, suppose $\Delta_d \alpha=0$. Then integration by parts gives
\begin{align*}
0 & = \langle \Delta_d\alpha,\alpha\rangle_{L^2} = \langle \Delta_{\del-\delbar}\alpha,\alpha\rangle_{L^2}
+ \langle \Delta_{\mu}\alpha,\alpha\rangle_{L^2} + \langle \Delta_{\mubar}\alpha,\alpha\rangle_{L^2} \\
& = \|(\del-\delbar)\alpha\|^2 + \|(\del^*-\delbar^*)\alpha\|^2 + \|\mu\alpha\|^2+ \|\mu^* \alpha\|^2+ \|\mubar\alpha\|^2+ \|\mubar^* \alpha\|^2,
\end{align*}
so $\alpha \in \bigcap_{P\in\mathcal{P}'}\ker P$ where $\mathcal{P}' = \{d, \mu,\del-\delbar,\mubar,d^*, \mu^*, \del^*-\delbar^*,\mubar^*\}$. The claim follows since $\del$, $\delbar$, $\del^*$, and $\delbar^*$ are linear combinations of operators in $\mathcal{P}'$. Now (b) follows from (a) and integration by parts.

To prove (c), we need to check that $\mathcal{\alpha}^{p,q}\in \mathcal{H}^{p,q}(M)$ for every $\alpha \in \mathcal{H}^k(M, \C)$.
Given one of the operators $P\in \mathcal{P}$, we have
$\sum_{p,q}P\alpha^{p,q}=0$.
Since $P \Omega^{p,q}(M)\cap P\Omega^{p',q'}(M)=\{0\}$ when $(p,q)\neq(p',q')$, we have $P\alpha^{p,q}=0$ for all $(p,q)$.  Finally, (d) follows from $\ol{P}(\ol{\alpha^{p,q}})=\ol{P(\alpha^{p,q})}$.
\end{proof}

\begin{cor}
Let $(M,g,J)$ be a compact nearly K\"ahler manifold. If $\ker(\Delta_{L_{\mu\omega}} - \Delta_{L_{\mubar\omega}})|_{\Omega^{p,q}(M)} = \{0\}$, then $\mathcal{H}^{p,q}(M) = \{0\}$.
\end{cor}
\begin{proof}
Since $M$ is compact, for $\alpha \in \mathcal{H}^{p,q}(M)$, we have $\alpha \in \ker\Delta_{\partial}$ and $\alpha \in \ker\Delta_{\delbar}$ by Theorem~\ref{thm:nk-hodge}(b), so $\alpha \in \ker(\Delta_{\del} - \Delta_{\delbar}) = \ker(\Delta_{L_{\mu\omega}} - \Delta_{L_{\mubar\omega}})$ by equation~\eqref{eqn:diff-lapl}. Therefore, if $\ker(\Delta_{L_{\mu\omega}} - \Delta_{L_{\mubar\omega}})|_{\Omega^{p,q}(M)} = \{0\}$, then $\mathcal{H}^{p,q}(M) = \{0\}$.
\end{proof}

\begin{rmk}
In particular, from Remark~\ref{rem:strict-nk-diff-lapl}, one deduces that for compact \emph{strict} nearly K\"ahler $6$-dimensional manifolds (that is those with $\lambda \neq 0$), we have $\mathcal{H}^{p,q}(M) = \{0\}$, except perhaps if $p=q$ or $p+q=3$. One further finds that $\mathcal{H}^{3,0}(M) = \{0\}$ because Proposition~\ref{prop:tor} implies 
$$ -3\mubar{\Theta} = \br{\Lambda}{L_{\mubar\omega}} \Theta = \Lambda(\mubar \omega\wedge\Theta) = \Lambda(3\lambda\ol{\Theta}\wedge\Theta) = 3\lambda\Lambda(\ol{\Theta}\wedge\Theta) \neq 0 $$
since $\ol{\Theta}\wedge\Theta$ is a multiple of $\omega^3$. In addition, we have $\mathcal{H}^{0,3}(M) = \{0\}$ by Theorem~\ref{thm:nk-hodge}(d).

These results were first proved in~\cite[Theorem 6.2]{Verbitsky}.
\end{rmk}

\section{Summary and discussion} \label{sec:conclusion}

Nearly K\"ahler manifolds in dimension $6$ are distinguished from those in higher dimensions, because in dimension $6$ they induce an $\SU{3}$ structure, while in general a nearly K\"ahler manifold of real dimension $2n$ does \emph{not} induce an $\SU{n}$ structure. Nevertheless, while Verbitsky in~\cite{Verbitsky} exploited this $\SU{3}$ structure extensively to establish his Hodge-theoretic results for nearly K\"ahler $6$-manifolds, we have shown in this paper that most (although not all) of Verbitsky's results still hold true for nearly K\"aher manifolds in any dimension, so they do not require an $\SU{n}$ structure.

In particular, for a compact nearly K\"ahler manifold in any dimension, from Theorem~\ref{thm:nk-hodge} we always have the following: if we define the ``Hodge numbers'' $h^{p,q}$ to be the dimensions of the spaces of the $\Delta_d$-harmonic forms of type $(p,q)$, then $h^{p,q} = h^{q,p}$ and $b^k = \sum_{p+q=k} h^{p,q}$. These are of course well-known in the K\"ahler case, where we can further identify $h^{p,q}$ with the space of $\Delta_{\delbar}$-harmonic forms of type $(p,q)$. Such an identification is not in general possible in the nearly K\"ahler case.

The result of Verbitsky that we were not able to generalize to higher dimensions was the fact that $\Delta_{L_{\mu\omega}} -\Delta_{L_{\mubar\omega}} = c_{p,q} \mathrm{Id}$ on forms of type $(p,q)$, where $c_{p,q}$ is a constant depending on $p,q$. This fact allowed Verbitsky to prove that for nearly K\"ahler $6$-manifolds, $h^{p,q} = 0$ except possibly when $p=q$ or $p+q=3$, although he then further exploited the $\SU{3}$ structure to show that $h^{3,0} = h^{0,3} = 0$ as well. (See our equation~\eqref{eqn:strict-2} and Remark~\ref{rem:strict-nk-diff-lapl}.)

There is a class of nearly K\"ahler manifolds in higher dimensions that shares some properties in common with nearly K\"ahler $6$-manifolds, in that they are equipped with a canonical complex-valued $3$-form $\gamma$ which satisfies a system of equations similar to~\eqref{eq:NK6}. These are nearly K\"ahler $(4n+2)$-manifolds that are twistor spaces over positive scalar curvature quaternionic-K\"ahler manifolds. Such spaces are equipped with a particular type of $\Sp{n} \! \cdot \! \U{1}$ structure. (See~\cite[Sections 4.1--4.2]{AKM}, particularly~\cite[equations following Theorem 4.8]{AKM}.) We note that $\Sp{1} \! \cdot \! \U{1}$ can be identified~\cite[Example 4.5]{AKM} with a subgroup of $\SU{3}$, but when $n > 1$, the analogous construction does not work to identify the group $\Sp{n} \! \cdot \! \U{1}$ with a subgroup of $\SU{2n+1}$.

It is natural to ask if there exists an identity that generalizes $(\Delta_{L_{\mu\omega}} -\Delta_{L_{\mubar\omega}})|_{\Omega^{p,q}} = c_{p,q} \mathrm{Id}$ for such spaces. Note that the real dimension of such spaces is $2(2n+1)$. If such an identity existed, and if $c_{p,q}$ was again of the form $\lambda (p-q) (2n+1-p-q)$ for some $\lambda \in \R \setminus \{0\}$, then we would also be able to deduce that $h^{p,q} = 0$ except possibly when $p=q$ or $p+q=2n+1$. It is not clear if we could then also exclude the ``edge'' cases $(p,q) = (2n+1,0), (0,2n+1)$, since that argument crucially uses the $\SU{3}$ structure in the $6$-dimensional case. Nevertheless, it would at least say that all the odd Betti numbers \emph{except possibly the middle one} must be zero. (Note that these twistor spaces also admit K\"ahler-Einstein metrics, so we know that all the odd Betti numbers must be even, but this would say that they are in fact zero.)

In fact, a stronger result than this is already known to be true. In~\cite[Theorem 6.6]{Salamon} it is shown that any quaternionic-K\"ahler manifold with positive scalar curvature is simply connected and has odd Betti numbers zero. This is equivalent to the odd Betti numbers of its twistor space being zero, because $S^2$ bundles over simply connected bases $B$ have the same Betti numbers as $B \times S^2$ (see~\cite[Corollary 4.41]{FOT}). Salamon's proof uses completely different ideas, including both the Kodaira and the Akizuki--Nakano vanishing theorems, and it also gives the vanishing of the middle Betti number.

This result of Salamon and the above discussion suggest that Hodge-theoretic arguments involving the graded commutators of various natural operators, as introduced by Verbitsky and extended by us in the present paper, may possibly be further extendible to such special $\SU{n} \! \cdot \! \U{1}$ structures in dimension $4n+2$. The authors investigated this situation, and found a complicated formula for $\Delta_{L_{\mu\omega}} - \Delta_{L_{\mubar\omega}}$ on $(p,q)$ forms of ``pure Lefschetz type'' with respect to the induced complex symplectic structure $\kappa = \kappa_2 + i \kappa_3$ and its conjugate $\bar{\kappa}$ on the horizontal distribution of the twistor space~\cite[Section 4.2]{AKM}. However, the authors were unable to exclude certain $(\kappa, \bar\kappa)$-primitive $(2,1)$-forms from this kernel using purely algebraic arguments. Thus, if such an approach is to succeed, some new ideas would be required.

Other interesting questions which arise from our investigations are the following:
\begin{itemize}
\item As mentioned in Section~\ref{sec:intro}, is it possible to prove formality of compact nearly K\"ahler manifolds in all dimensions by using these general nearly K\"ahler identities to establish an analogue of the $\partial \ol{\partial}$-lemma, without resorting to the case-by-case argument in~\cite{AmannTaimanov} that relies on the Nagy classification?
\item Can we generalize the Hodge decomposition to Riemannian manifolds admitting a metric-compatible connection whose torsion is totally skew-symmetric and closed? (Such objects are playing an increasingly important role in physics and in generalized geometry. See for example,~\cite{Agricola, Mario-Jeff}, for more details.)
\end{itemize}

\addcontentsline{toc}{section}{References}

\bibliographystyle{plain}
\bibliography{NK-Hodge.bib}

\end{document}